\documentclass[table]{amsart}
\usepackage[margin=1.0in]{geometry}
\usepackage{amscd,amsmath,amsxtra,amsthm,amssymb,stmaryrd,xr,mathrsfs,mathtools,enumerate,commath, comment, enumitem, graphicx}
\usepackage{amsfonts, amssymb, amsthm, amsmath, braket, hyperref, mathtools, comment, mathrsfs, enumitem, cleveref}
\usepackage{stmaryrd}
\usepackage{orcidlink}
\usepackage{multirow}
\usepackage{xcolor}
\usepackage{commath}
\usepackage{comment}
\usepackage{tikz-cd}
\usepackage{tkz-graph}
\usepackage{colonequals}
\usepackage{hyperref}

\usepackage{longtable} % for 'longtable' environment
\usepackage{pdflscape} % for 'landscape' environment
\usepackage{booktabs}
\usepackage{hyperref}
\definecolor{vegasgold}{rgb}{0.77, 0.7, 0.35}
\definecolor{darkgoldenrod}{rgb}{0.72, 0.53, 0.04}
\definecolor{gold(metallic)}{rgb}{0.83, 0.69, 0.22}
\hypersetup{
 colorlinks=true,
 linkcolor=darkgoldenrod,
 filecolor=brown,      
 urlcolor=gold(metallic),
 citecolor=darkgoldenrod,
 }
\newtheorem{lthm}{Theorem}

\usepackage[all,cmtip]{xy}
\DeclareFontFamily{U}{wncy}{}
\DeclareFontShape{U}{wncy}{m}{n}{<->wncyr10}{}
\DeclareSymbolFont{mcy}{U}{wncy}{m}{n}
\DeclareMathSymbol{\Sh}{\mathord}{mcy}{"58}
\usepackage[T2A,T1]{fontenc}
\usepackage[OT2,T1]{fontenc}

\usepackage{tikz}
\usetikzlibrary{shapes.geometric}
\usetikzlibrary{decorations.markings}
\tikzset{every loop/.style={min distance=10mm,looseness=10}}
\tikzstyle{vertex}=[auto=left,circle,minimum size=1pt,inner sep=0pt]

\newtheorem{theorem}{Theorem}[section]
\newtheorem{lemma}[theorem]{Lemma}
\newtheorem{ass}[theorem]{Assumption}
\newtheorem*{theorem*}{Theorem}
\newtheorem*{ass*}{Assumption}
\newtheorem{definition}[theorem]{Definition}

\newtheorem{remark}[theorem]{Remark}

\newtheorem{proposition}[theorem]{Proposition}

\newcommand{\cF}{\mathcal{F}}

\newcommand{\cE}{\mathcal{E}}

\newcommand{\Z}{\mathbb{Z}}

\newcommand{\Q}{\mathbb{Q}}
\newcommand{\F}{\mathbb{F}}

\newcommand{\cC}{\mathcal{C}}

\newcommand{\cS}{\mathcal{S}}

\newcommand{\op}[1]{\operatorname{#1}}

\numberwithin{equation}{section}

\begin{document}

\title[Counting entanglements]{Counting elliptic curves with prescribed entanglements}

\author[Z.~Couvillion]{Zachary Couvillon\, \orcidlink{0000-0003-2687-3718}}
\address[Couvillion]{Cornell University \\ Department of Mathematics \\
  Malott Hall, Ithaca, NY 14853-4201 USA} 
\email{zjc28@cornell.edu}

\author[A.~Ray]{Anwesh Ray\, \orcidlink{0000-0001-6946-1559}}
\address[Ray]{Chennai Mathematical Institute, H1, SIPCOT IT Park, Kelambakkam, Siruseri, Tamil Nadu 603103, India}
\email{anwesh@cmi.ac.in}


\maketitle

\begin{abstract}
We establish asymptotic lower bounds for the number of elliptic curves over \( \mathbb{Q} \) with prescribed entanglement of division fields, ordered by naive height. Such elliptic curves are obtained as $1$-parameter families arising from certain genus $0$ modular curves. We apply techniques from the geometry of numbers and sieve methods to prove that the number of elliptic curves with unexplained entanglements $\mathbb{Q}(E[2]) \cap \mathbb{Q}(E[3]) \neq \mathbb{Q}$ and $\mathbb{Q}(E[2]) \cap \mathbb{Q}(E[5]) \neq \mathbb{Q}$ and naive height $\leq X$,
grows as \( \gg X^{1/9} \) and \( \gg X^{1/12} \), respectively.
\end{abstract}

\section{Introduction}
Let \( E \) be an elliptic curve over \( \mathbb{Q} \) with torsion subgroup \( E(\mathbb{Q})_{\mathrm{tors}} \) of the Mordell–Weil group. Mazur famously showed that there are precisely 15 possible torsion subgroups for elliptic curves over \( \mathbb{Q} \), namely  
\begin{equation} \label{torsion groups}  
\begin{split}  
\mathbb{Z} / N \mathbb{Z} & \quad \text{for } 1 \leq N \leq 10 \text{ or } N = 12, \\  
\mathbb{Z} / 2 \mathbb{Z} \times \mathbb{Z} / N \mathbb{Z} & \quad \text{for } N = 2, 4, 6, 8.  
\end{split}  
\end{equation}  
Each of these groups arises infinitely often, as they correspond to the cases where the associated modular curves have genus \( 0 \). A natural question is how frequently these groups occur when elliptic curves are ordered by their \emph{naive height}. Every elliptic curve \( E \) over \( \mathbb{Q} \) admits a unique Weierstrass equation of the form  
\[
y^2 = x^3 + A x + B,  
\]  
where \( A, B \in \mathbb{Z} \) and \( \gcd(A^3, B^2) \) is not divisible by any twelfth power. The naive height of \( E \) is defined by  
\[
\op{ht}(E) := \max(4|A|^3, 27|B|^2).  
\]  
For any group \( G \) appearing in \eqref{torsion groups}, Harron and Snowden \cite{harronsnowden} showed that  
\[
\#\{E / \mathbb{Q} \mid \op{ht}(E) \leq X, \, E(\mathbb{Q})_{\mathrm{tors}} \simeq G\} \gg X^{1/d(G)},  
\]  
where \( d(G) > 0 \) is explicitly given in Table 1 of \emph{loc. cit.} Similar results were obtained by Cullinan, Kenny, and Voight \cite{CKV}. The case \( G = 0 \) follows from work of Duke \cite{Dukenoexceptional}, who showed that almost all elliptic curves have no \emph{exceptional primes}. 

\par Let $\rho_{E, n} : \operatorname{Gal}(\overline{\mathbb{Q}} / \mathbb{Q}) \rightarrow \operatorname{Aut}(E[n]) \xrightarrow{\sim} \operatorname{GL}_2( \mathbb{Z} / n \mathbb{Z})$
denote the Galois representation on the \( n \)-torsion of \( E \). The elliptic curve $E$ admits a cyclic $n$-isogeny precisely when (for a suitable choice of basis for $E[n]$) the image of $\rho_{E,n}$ consists of upper triangular matrices. Asymptotics for the number of elliptic curves for prescribed $n$-cyclic isogeny have been studied using the parametrization from the modular curve $Y_0(n)$, which has genus $0$ for $1\leq n \leq 10$ and $n=12,13,16,18,25$. The study of the asymptotic for \[\#\{E_{/\Q}\mid \rho_{E,n}\text{ is reducible and }H(E)\leq X\}\] as $X\rightarrow \infty$ has garnered significant interest.
For \( n = 3 \), an exact asymptotic formula is established by Pizzo--Pomerance--Voight \cite{PPV}, and for \( n = 4 \) by Pomerance--Schaefer \cite{PS}. For \( n = 7, 10, 13, \) and \( 25 \), Boggess--Sankar \cite{Sankar} obtain asymptotic lower bounds, while for \( n = 7 \), Molnar--Voight \cite{MV} obtain the precise asymptotic.

\par The action of \( \operatorname{Gal}(\overline{\mathbb{Q}} / \mathbb{Q}) \) on the Tate module \( T(E) = \varprojlim_n E[n] \) induces an adelic Galois representation  
\[
\widehat{\rho}_E : \operatorname{Gal}(\overline{\mathbb{Q}} / \mathbb{Q}) \rightarrow \operatorname{Aut}(T(E)) \cong \operatorname{GL}(2, \widehat{\mathbb{Z}}).
\]  
Understanding the image of \( \widehat{\rho}_E \) is a central problem in what is known as Mazur's Program B, following \cite{Mazurratpoints} which seeks to classify all Galois images for elliptic curves. As part of this program, it is natural to study \emph{entanglements} of division fields.  

\par Let \( a \) and \( b \) be positive integers with \( d = \gcd(a, b) \). It is easy to see that
\[
\mathbb{Q}(E[a]) \cap \mathbb{Q}(E[b]) \supseteq \mathbb{Q}(E[d]).  
\]  
An \((a, b)\)-entanglement occurs when this containment is strict, i.e.,  
\[
K := \mathbb{Q}(E[a]) \cap \mathbb{Q}(E[b]) \neq \mathbb{Q}(E[d]).  
\]  
An entanglement is said to be \emph{unexplained} if it cannot be explained by the Weil pairing on $E$, for further details, see Definition \ref{unexplained defn}. Such entanglements are mysterious and their properties are of significant interest. The \emph{type} of the entanglement is the isomorphism class of \( \operatorname{Gal}(K / \mathbb{Q}(E[d])) \). Given an unexplained entanglement \( \mathfrak{T} \), specified by a tuple $(a,b)$ and a group $G$, let \( N_{\mathfrak{T}}(X) \) denote the number of isomorphism classes of elliptic curves \( E_{/\mathbb{Q}} \) with \( \op{ht}(E) \leq X \) exhibiting \( \mathfrak{T} \) as an entanglement (i.e. $E$ has an $(a,b)$-entanglement of type $G$). Below is the main result of this article.
\begin{lthm}[Theorem \ref{(2,3) theorem section 4} and \ref{(2, 5) theorem section 4}]\label{thm a}
    With respect to the notation above, the following assertions hold:
    \begin{enumerate}
        \item \( N_{\left((2,3),\mathbb{Z}/2\mathbb{Z}\right)}(X) \gg X^{1/9} \);  
        \item \( N_{\left((2,5),\mathbb{Z}/2\mathbb{Z}\right)}(X) \gg X^{1/12}. \)  
    \end{enumerate}
\end{lthm}

\par We now provide a brief yet comprehensive overview of the method of proof. Building on the results of \cite{danielsmorrow}, we exploit the existence of explicit families of elliptic curves of the form \( y^2 = x^3 + f(t) x + g(t) \) with prescribed unexplained entanglement. Let \( \mathcal{F}_1 \) denote the family of elliptic curves for which:
\[
\begin{aligned}
f(t) &:= (-3t^2 - 36)c(t), \\
g(t) &:= (2t^5 + 72t^3 + 864t)c(t), \\
c(t) &:= \gcd(f(t), g(t)) = t^4 + 36t^2 + 432,  
\end{aligned}
\]  
and let \( \mathcal{F}_2 \) denote the family defined by: 
\[
\begin{aligned}
f(t) &:= (-3 t^{4} - 30 t^{2} - 15)c(t), \\
g(t) &:= 2(t^{4} + 4 t^{2} - 1)(t^{4} + 22 t^{2} + 125)c(t), \\
c(t) &:= \gcd(f(t), g(t)) = t^{4} + 22 t^{2} + 125.
\end{aligned}
\]  
Setting \( \mathfrak{T}_1 := ((2,3), \mathbb{Z}/2\mathbb{Z}) \) and \( \mathfrak{T}_2 := ((2,5), \mathbb{Z}/2\mathbb{Z}) \), the family \( \mathcal{F}_i \) consists of elliptic curves exhibiting the entanglement \( \mathfrak{T}_i \). Let \( A(u, v) \), \( B(u, v) \), and \( C(u, v) \) be nonconstant homogeneous polynomials in \( \mathbb{Z}[u, v] \).  
\begin{ass}\label{main ass} We assume that: 
\begin{enumerate}
    \item[(i)] The degrees of \( A \) and \( B \) are related by the equality \( 3\deg A = 2\deg B \). We set \( d := 3\deg A = 2\deg B \).  
    \item[(ii)] The polynomials \( A(u, v) \) and \( B(u, v) \) have no common real roots in \( \mathbb{P}^1(\mathbb{C}) \).  
    \item[(iii)] The discriminant is non-vanishing, i.e., \( 4A^3 + 27B^2 \neq 0 \).  
    \item[(iv)] The greatest common divisor of \( A(u, v) \) and \( B(u, v) \), denoted by \( C(u, v) := \gcd(A, B) \), is a non-constant, squarefree polynomial that is not divisible by \( u \) or \( v \). Moreover, its degree satisfies \( r := \deg(C) \leq 4 \).  
\end{enumerate}
\end{ass}  
\par In our applications, the polynomials \( A \), \( B \), and \( C \) will be the homogenizations of \( f \), \( g \), and \( c \), respectively. Consider the family of elliptic curves with minimal Weierstrass model given by \( \mathcal{E}_{a,b} := E_{A(a,b), B(a,b)} \). Defining  
\[
\mathcal{C}(X) = \left\{ \mathcal{E}_{a, b} \mid \mathcal{E}_{a, b} \text{ is minimal and } \op{ht}(\mathcal{E}_{a, b}) \leq X \right\},
\]  
we establish the following general result.  
\begin{lthm}[Theorem \ref{main thm section 3}]\label{thm b}
     With respect to the notation above, we have  
   \[
   \# \mathcal{C}(X) \sim \prod_\ell \mathfrak{d}_\ell \operatorname{Area}(\mathcal{R}) X^{2/d},
   \]  
   where \( \mathcal{R} \) is the bounded region defined by  
   \[
   \mathcal{R} := \{(u, v) \in \mathbb{R}^2 \mid v > 0, 4|A(u,v)|^3\leq 1\;\text{and }\;27|B(u,v)|^2 \leq 1\},
   \]  
   and \( \mathfrak{d}_\ell \in (0,1] \) is an explicit constant satisfying 
   \[
   \mathfrak{d}_\ell \geq \left(1 - \frac{1}{\ell^2} - \frac{(r + \ell)}{\ell^8}\right)
   \]  
   for all but finitely many primes.
\end{lthm}
The product $\prod_{\ell} \mathfrak{d}_\ell$ is non-vanishing via a straightforward argument, see Remark \ref{non-vanishing of product}. The proof of this result combines analytic techniques, sieve methods, and an application of Davenport's lemma from the geometry of numbers. The overall approach differs significantly from those in previous works, although it bears some resemblance to the methods of Harron--Snowden and Molnar--Voight. The key distinction between our technique and that of Harron and Snowden lies in the fact that, in our setting, the polynomials \( f(t) \) and \( g(t) \) share a non-trivial common factor. In contrast, the method of Molnar and Voight applies only to a more restricted situation and does not immediately extend to the generality of Theorem B. Moreover, the geometry of numbers framework we adopt is considerably more flexible, as it could naturally adapt to settings involving several variables and higher-dimensional parameter spaces, whereas analytic techniques relying on iterated sums and sieve-theoretic estimates would soon become much more complicated.

\par Theorem \ref{thm a} is then deduced by applying Theorem \ref{thm b} to the families of elliptic curves \( \mathcal{F}_1 \) and \( \mathcal{F}_2 \). However, our approach does not yield a precise asymptotic formula, as the families \( \mathcal{F}_i \) do not encompass all elliptic curves \( E_{/\mathbb{Q}} \) with entanglement type \( \mathfrak{T}_i \). This incompleteness arises from the fact that the parametrization in terms of \( \mathcal{F}_1 \) and \( \mathcal{F}_2 \) excludes certain subfamilies that arise from quadratic twists. Thus, we are not counting all $\Q$-isomorphism classes of elliptic curves with $j$-invariants given in \cite{danielsmorrow}.

\par A crucial aspect of our method is the assumption that the polynomial \( c(t) \) has no real roots. This condition is essential in the application of Davenport's lemma, as it guarantees that the region \( \mathcal{R} \) appearing in the geometry-of-numbers argument remains bounded. When \( c(t) \) admits real roots, the region becomes unbounded, rendering the application of Davenport's lemma ineffective. Consequently, this restriction prevents us from extending our strategy to certain other entanglement families, such as those arising in the parametrization of Morrow and Daniels, where the geometry of the corresponding regions is more complicated. 

\par Nevertheless, the results of this article offer new insights into the arithmetic statistics of elliptic curves, and we expect that they will stimulate further investigation into related families and their distribution. In particular, it would be interesting to determine whether a modification of our sieve-theoretic framework could be adapted to cover the missing families, potentially leading to a more comprehensive asymptotic count for entangled elliptic curves.

\section*{Statements and Declarations}

\subsection*{Conflict of interest} The authors have no conflicts of interests to declare.

\subsection*{Data Availability} There is no data associated to the results of this manuscript.

\section{Galois images and entanglements}

\par Throughout this article, we shall fix an algebraic closure $\overline{\Q}$ of $\Q$ and let $\op{G}_{\Q}$ denote the absolute Galois group $\op{Gal}(\overline{\Q}/\Q)$. Let $E:y^2=x^3+Ax+B$ be an elliptic curve defined over $\Q$. Given $n\in \Z_{\geq 1}$, let $\rho_{E, n}:\op{G}_{\Q}\rightarrow \op{GL}_2(\Z/n\Z)$ be the Galois representation associated to $E[n]$. Denote by $G_{E,n}$ the image of $\rho_{E,n}$ and $\Q(E[n]):=\overline{\Q}^{\rho_{E,n}}$. The field $\Q(E[n])$ is called the $n$-division field of $E$. The map $\rho_{E,n}$ gives us a natural identification of $\op{Gal}(\Q(E[n])/\Q)$ with the Galois image $G_{E,n}$. Let $\chi_n$ be the mod-$n$ cyclotomic character. As a consequence of the Weil--pairing, $\det \rho_{E,n}=\chi_n$ and therefore, $\Q(\mu_n)$ is contained in $\Q(E[n])$. If $E$ is a non-CM elliptic curve then Serre's open image theorem \cite{Serre} asserts that the index $d_E:=[\op{GL}_2(\widehat{\Z}): \op{image}\widehat{\rho}_E]$ is finite. Serre moreover showed that $d_E$ is always divisible by $2$ and that, in fact, the image of $\widehat{\rho}_E$ is contained in a prescribed index-two subgroup $H_E$ of $\op{GL}_2(\widehat{\Z})$. Elliptic curves for which $d_E=2$ are called Serre curves. Building upon previous work of Duke \cite{Dukenoexceptional}, Jones \cite{Jones} showed that when ordered by \emph{height}, most elliptic curves $E_{/\Q}$ are Serre curves. Serre curves have the property that for every prime $\ell$, the Galois representation $\rho_{E, \ell} : \op{G}_{\Q}\rightarrow \op{GL}_2(\Z/\ell\Z)$ is surjective.

\par The image of \(\widehat{\rho}_E\) could be strictly smaller than \(H_E\) in various ways. The first arises from the existence of an \emph{exceptional} prime, i.e., a prime \(\ell\) for which the Galois representation \(\rho_{E, \ell}\) is not surjective. For non-exceptional primes \(\ell\), the Galois group \(\op{Gal}(\Q(E[\ell])/\Q)\) is isomorphic to \(\op{GL}_2(\Z/\ell\Z)\). A conjecture of Serre \cite[p. 399]{Serreconjpaper} and Zywina \cite{zywina2015possible} predicts that if \(E\) is a non-CM elliptic curve over \(\Q\), then there are no exceptional primes \(\ell > 37\). In fact, Zywina provides a more refined conjecture \cite[Conjecture 1.1]{ZywinaBLMS} describing precisely when \(\rho_{E,\ell}\) is surjective, along with a practical algorithm for computing the image. Another obstruction to surjectivity arises from the presence of an entanglement, as defined in the introduction. \par These entanglements are further subdivided into two categories: explained and unexplained. Let us illustrate the difference through an example. Let \(\Delta_E\) denote the minimal discriminant of \(E\), and assume that \(\Delta_E\) is not a square. It is straightforward to verify that \(\mathbb{Q}(\sqrt{\Delta_E}) \subseteq \mathbb{Q}(E[2])\). Now, choose \( m > 2 \) such that \(\mathbb{Q}(\sqrt{\Delta_E})\) is contained in \(\mathbb{Q}(\mu_m)\). Since \(\mathbb{Q}(\mu_m)\) is itself contained in \(\mathbb{Q}(E[m])\), it follows that  
\[
\mathbb{Q}(\sqrt{\Delta_E})\subseteq \mathbb{Q}(E[2]) \cap \mathbb{Q}(E[m]) .
\]
If the minimal such \( m \) is odd, we obtain a \((2, m)\)-entanglement of type \(\mathbb{Z}/2\mathbb{Z}\). When \( m \) is even, one instead finds either a \((4, m/4)\)-entanglement of type \(\mathbb{Z}/2\mathbb{Z}\), arising from  
\[
\mathbb{Q}\left(\sqrt{-\Delta_E}\right) \subseteq \mathbb{Q}(E[4]) \cap \mathbb{Q}(E[m/4]),
\]  
or an \((8, m/8)\)-entanglement of type \(\mathbb{Z}/2\mathbb{Z}\), arising from  
\[
\mathbb{Q}\left(\sqrt{\pm 2 \Delta_E}\right) \subseteq \mathbb{Q}(E[8]) \cap \mathbb{Q}(E[m/8]).
\]  
These are referred to as \emph{explained entanglements}, following Daniels and Morrow \cite{danielsmorrow}. Let $E: y^{2}=x^{3}-x^{2}-1033 x-12438
$ with Cremona label \href{https://www.lmfdb.org/EllipticCurve/Q/100a3/}{100a3}. Then \cite[Example 1.3 on p.~829]{danielsmorrow} shows that $\Q(E[3])\cap \Q(E[2])=\Q(\sqrt{5})$. In this case, there is an entanglement $((2,3), \Z/2\Z)$ which cannot be explained by the Weil pairing or the Kronecker--Weber theorem. Such entanglements are more interesting and called \emph{unexplained entanglements}. We give a formal definition later in this section. 
\begin{definition}\label{defn admissible}A subgroup \( G \) of \( \mathrm{GL}_2(\mathbb{Z} / N \mathbb{Z}) \) is said to be \emph{admissible} if it satisfies the following conditions: 
\begin{enumerate}
    \item[(a)] \( G \) is a proper subgroup of \( \mathrm{GL}_2(\mathbb{Z} / N \mathbb{Z}) \), i.e., \( G \neq \mathrm{GL}_2(\mathbb{Z} / N \mathbb{Z}) \),  
    \item[(b)] The determinant map induces an equality \( \operatorname{det}(G) = (\mathbb{Z} / N \mathbb{Z})^{\times} \),  
    \item[(c)] \( G \) contains an element with trace \( 0 \) and determinant \( -1 \) that fixes a point in \( (\mathbb{Z} / N \mathbb{Z})^2 \) of exact order \( N \).
\end{enumerate}  
\end{definition}
Let \( E \) be an elliptic curve over \( \mathbb{Q} \) the mod-\( N \) Galois representation is not surjective. Then \( \rho_{E, N}(G_{\mathbb{Q}}) \) is an admissible subgroup of \( \mathrm{GL}_2(\mathbb{Z} / N \mathbb{Z}) \), see \cite[Proposition 2.2]{zywina2015possible} for further details. Let \( G \) be a subgroup of \( \mathrm{GL}_{2}(\mathbb{Z} / n \mathbb{Z}) \) for some \( n \geq 2 \) whose determinant map is surjective. Consider two divisors \( a \) and \( b \) of \( n \) with \( a < b \), and define \( c := \operatorname{lcm}(a, b) \) and \( d := \operatorname{gcd}(a, b) \). Let \( \pi_c: \mathrm{GL}_{2}(\mathbb{Z} / n \mathbb{Z}) \to \mathrm{GL}_{2}(\mathbb{Z} / c \mathbb{Z}) \) be the canonical reduction map, and set \( G_c = \pi_c(G) \). We then introduce the following reduction maps and associated normal subgroups of \( G_c \):  
\[
\begin{array}{ll}
\pi_a: \mathrm{GL}_{2}(\mathbb{Z} / c \mathbb{Z}) \to \mathrm{GL}_{2}(\mathbb{Z} / a \mathbb{Z}), & N_a = \ker(\pi_a) \cap G_c, \\
\pi_b: \mathrm{GL}_{2}(\mathbb{Z} / c \mathbb{Z}) \to \mathrm{GL}_{2}(\mathbb{Z} / b \mathbb{Z}), & N_b = \ker(\pi_b) \cap G_c, \\
\pi_d: \mathrm{GL}_{2}(\mathbb{Z} / c \mathbb{Z}) \to \mathrm{GL}_{2}(\mathbb{Z} / d \mathbb{Z}), & N_d = \ker(\pi_d) \cap G_c.
\end{array}
\]
\begin{definition}\label{defn of entanglement} With respect to notation above, we say that $G$ represents an $(a, b)$-entanglement if
$
\left\langle N_{a}, N_{b}\right\rangle \subsetneq N_{d} .
$ The type of the entanglement is the isomorphism type of the group $N_{d} /\left\langle N_{a}, N_{b}\right\rangle$.
\end{definition}
Consider the set  
\[
\mathcal{T}_{G} = \{((a, b), H) \mid G \text{ represents an } (a, b) \text{-entanglement of type } H\}
\]  
and declare that
$\left(\left(a_{1}, b_{1}\right), H_{1}\right) \leqslant \left(\left(a_{2}, b_{2}\right), H_{2}\right)$
if one of the following conditions hold. 
\begin{enumerate}
\item[(a)] The groups \( H_1 \) and \( H_2 \) are isomorphic, and either \( a_2 \) divides \( a_1 \) and \( b_2 \) divides \( b_1 \), or \( b_2 \) divides \( a_1 \) and \( a_2 \) divides \( b_1 \).  
\item[(b)] The group \( H_1 \) is isomorphic to a quotient of \( H_2 \), and either \( a_1 \) divides \( a_2 \) and \( b_1 \) divides \( b_2 \), or \( b_1 \) divides \( a_2 \) and \( a_1 \) divides \( b_2 \).  
\end{enumerate}
One says that \( G \) represents a \emph{primitive} \( (a, b) \)-entanglement of type \( H \) if \( ((a, b), H) \) is the unique maximal element of \( \mathcal{T}_{G} \) and \( n = \operatorname{lcm}(a, b) \).

\begin{definition}\label{unexplained defn}
    A group \( G \) exhibits an explained \( (a,b) \)-entanglement of type \( T \) if it satisfies the conditions for a primitive \( (a,b) \)-entanglement of type \( T \) and the equality  
\[
\left[(\mathbb{Z} / c \mathbb{Z})^{\times}: \operatorname{det}\left(N_{a, b}(G)\right)\right] = \left[G_{c}: N_{a, b}(G)\right]
\]  
holds. Equivalently, this occurs when the kernel \( N \) of \( \operatorname{det}: G_c \to (\mathbb{Z} / c \mathbb{Z})^{\times} \) is contained in \( \langle N_a, N_b \rangle \).  

In contrast, \( G \) exhibits an unexplained \( (a,b) \)-entanglement of type \( T \) if it satisfies the conditions for a primitive \( (a,b) \)-entanglement of type \( T \) but the above equality fails.
\end{definition}

\begin{theorem}{\cite[Theorem A]{danielsmorrow}}  
There exist precisely nine pairs \(((p, q), T)\), where \( p \) and \( q \) are distinct prime numbers with \( p < q \), and \( T \) is a finite group, such that there are infinitely many elliptic curves over \(\mathbb{Q}\) that are pairwise non-isomorphic, and exhibit an \emph{unexplained} \((p, q)\)-entanglement of type \(T\).  
\end{theorem}
For further details we refer to \cite[Appendix A]{danielsmorrow}.
\section{Density results}
\par Let $E_{/\Q}$ be an elliptic curve defined by a Weierstrass equation in \emph{short form}:
\[E: y^2 = x^3 + Ax + B,\] where $A, B\in \Z$. The \emph{height} of its Weierstrass model is defined by $\op{H}(A,B) := \max\left\{ |4A^3|,|27B^2|\right\}$.

\begin{definition}\label{defn of minimality defect}
    The \emph{minimality defect} $\op{md}(A,B)$ is the largest \( d \in \mathbb{Z}_{>0} \) such that \( d^4 \mid A \) and \( d^6 \mid B \), and the \emph{height} of \( E \) is  
\[
\op{ht}(E) = \frac{H(A,B)}{\op{md}(A,B)^{12}}.
\]  
\end{definition}
Each elliptic curve has a unique \emph{globally minimal model}, i.e., a model for which $\op{md}(A, B)=1$.

\par Choose non-constant homogeneous polynomials \( A(u,v), B(u,v) \in \mathbb{Z}[u,v] \) satisfying the conditions of Assumption \ref{main ass}. Given a homogeneous polynomial \( F(u,v) \in \mathbb{Z}[u,v] \), we write it in the form  
\[
F(u,v) = \alpha \prod_{i=1}^t (u - \alpha_i v)^{n_i}
\]  
as a product of linear factors, where \( \alpha \in \mathbb{Z}_{\geq 1} \) and \( \alpha_i \in \mathbb{P}^1(\mathbb{C}) \). By convention, we interpret \( (u - \infty v) \) as \( v \).\\
 According to condition (ii) above, $C(u,v)$ has no roots in $\mathbb{P}^1(\mathbb{R})$. Assume therefore upon replacing $C(u,v)$ with $-C(u,v)$ if necessary that $C(u,v)\geq 0$ for all $(u,v)\in \mathbb{R}^2$. We write $A(u, v)=C(u,v) A_0(u,v)$ and $B(u,v)=C(u, v) B_0(u,v)$.\\

%If we let $P(t) = A(t,1)$ and $Q(t) = B(t,1)$, then being able to arrange for condition (i) of the assumption above seems to depend on the $j$-invariant of $C_t: y^2 = x^3 + P(t)x + Q(t)$ over $\Q(t)$. More precisely, the relation $3 \deg P(t) = 2 \deg Q(t)$ is invariant under quadratic twists, since $2\deg C(t)^3Q(t) = 6\deg C(t)+2\deg Q(t) = 3\deg C(t)^2+ 3\deg P(t) = 3\deg C(t)^2P(T)$, so in particular if (i) holds for $C_t$ then it will hold for
%$$y^2 = x^3  -\frac{27j(t)}{j(t) - 1728}x + \frac{54j(t)}{j(t) - 1728}  $$
%where $j(t)$ is the $j$-invariant of $E_t$. For example, we see that condition (i) holds if $j(t)$ is a nonconstant polynomial, in which case the rational functions are degree $0$.

\par Let \(\cF\) denote the set of all pairs of integers \((a, b)\) with \(b > 0\). For \((a, b) \in \cF\), define
\[
H(a, b) := H(A(a, b), B(a, b)) \quad \text{and} \quad m(a, b) := \op{md}(A(a, b), B(a, b)).
\]

Let \(\cC \subseteq \cF\) denote the subset consisting of pairs \((a, b)\) satisfying:
\begin{enumerate}
    \item \(\gcd(a, b) = 1\),
    \item \(4A(a, b)^3 + 27B(a, b)^2 \neq 0\),
    \item \(m(a, b) = 1\).
\end{enumerate}

For each \((a, b) \in \cC\), define the elliptic curve
\[
\cE_{a, b} = E_{A(a, b), B(a, b)} : y^2 = x^3 + A(a, b) x + B(a, b).
\]
Note that
\[
H(a, b) = H(\cE_{a, b}) \quad \text{and} \quad m(a, b) = \op{md}(\cE_{a, b}) = 1.
\]
Since \(m(a, b) = 1\), the curve \(\cE_{a, b}\) is globally minimal.

Given a subset \(\mathcal{S} \subseteq \cF\) and a positive real number \(X\), define
\[
\mathcal{S}(X) := \{(a, b) \in \mathcal{S} \mid H(a, b) \leq X\}.
\]
In particular, the set \(\cC(X)\) is identified with the set of globally minimal elliptic curves of the form $\cE_{a,b}=E_{A(a,b), B(a,b)}$ of height at most \(X\), namely
\[
\cC(X) = \{\cE_{a, b} \mid \cE_{a, b} \text{ is minimal and } H(\cE_{a, b}) \leq X\}.
\]
Let $\mathcal{R}(X)$ denote the region $\mathcal{R}(X):=\{(u, v)\in \mathbb{R}^2\mid v> 0, H(u,v)\leq X\}$,
and set $\mathcal{R}:=\mathcal{R}(1)$. 
\begin{lemma}\label{bounded area lemma}
    With respect to notation above, the set $\mathcal{R}(X)$ is compact and $\op{Area}(\mathcal{R}(X))=X^{2/d} \op{Area}(\mathcal{R})$.
\end{lemma}
\begin{proof}
    It is clear that $\mathcal{R}(X)$ is closed; we show that $\mathcal{R}(X)$ is bounded. We write $A(u,v)=c_A\prod_{\alpha\in R_A} (u-\alpha v)^{n_\alpha}$ and $B(u,v)=c_B\prod_{\alpha\in R_B} (u-\alpha v)^{m_\alpha}$, where $R_A$ (resp. $R_B$) are the roots of $A$ (resp. $B$). Assume without loss of generality that $X\geq 27$. For $(u,v)\in \mathcal{R}(X)$ one has that $|A(u,v)|\leq X_1$ and $|B(u,v)|\leq X_2$ where $X_1:=(X/4)^{1/3}$ and $X_2:=(X/27)^{1/2}$. Note that $X_1\geq 1$ and $X_2\geq 1$. Given $z=(u,v)\in \mathbb{R}^2$ set $|z|:=\op{max}\{|u|, |v|\}$ and suppose by way of contradiction that $\mathcal{R}(X)$ is unbounded. Then, there is an infinite sequence $z_n=(u_n, v_n)\in \mathcal{R}(X)$ such that $|z_n|\rightarrow \infty$. Since $|A(u_n, v_n)|\leq X_1$ and $X_1\geq 1$, it follows that there is a root $\alpha$ of $A(u,v)$ such that $|u_n-\alpha v_n|\leq X_1$. 
    
    \par First, consider the case when for infinitely many values of $n$, $|v_n|\geq |u_n|$. Upon passing to a subsequence, assume without loss of generality that for all $n$, $|v_n|\geq |u_n|$. In particular, $|v_n|\rightarrow \infty$. Assume that $\alpha\neq \infty$, then we have that 
    \[|u_n/v_n-\alpha|\leq X_1/|v_n|.\] Therefore, for the root $\alpha$ of $A(u,v)$, we have that $|u_n/v_n-\alpha|\rightarrow 0$. In particular, $\alpha$ must be a real number. The same reasoning shows that there is a real root $\beta$ of $B(u, v)$ such that $|u_n/v_n-\beta|\rightarrow 0$. Being the limit of the same sequence, we get that $\alpha=\beta$. Since it is assumed that $A$ and $B$ have no common real root, this gives a contradiction. If $\alpha=\infty$, we find that $|v_n|\leq X_1$, which is a contradiction since $|v_n|\rightarrow \infty$. 
    \par Consider now the case when all but finitely many values of $n$, $|u_n|\geq |v_n|$. Assume without loss of generality that $|u_n|\geq |v_n|$ for all $n$. The symmetric argument with variables $u$ and $v$ interchanged gives a contradiction. Therefore, we have shown that $\mathcal{R}(X)$ is compact. 
    
    \par It is easy to see that $\mathcal{R}(X)=X^{1/d}\cdot \mathcal{R}$. Thus, $\op{Area}(\mathcal{R}(X))=X^{2/d} \op{Area}(\mathcal{R})$. This completes the proof.
\end{proof}

We recall the Principle of Lipschitz, which is also known as Davenport's lemma.
\begin{lemma}[Principle of Lipschitz/ Davenport's lemma]\label{davenport} Let \(\mathcal{R} \subseteq \mathbb{R}^2\) be a bounded region with a rectifiable boundary \(\partial \mathcal{R}\). Then, the number of integer lattice points contained in \(\mathcal{R}\) satisfies  
\[
\#(\mathcal{R} \cap \mathbb{Z}^2) = \op{Area}(\mathcal{R}) + O(\operatorname{len}(\partial \mathcal{R})),
\]
where the implicit constant depends only on the similarity class of \(\mathcal{R}\) and is independent of its size, orientation, or location in the plane \(\mathbb{R}^2\).
\end{lemma}

Lemma \ref{bounded area lemma} and Lemma \ref{davenport} together imply that
\begin{equation}\label{FX lemma}
    \# \mathcal{F}(X)\sim X^{2/d} \op{Area}(\mathcal{R}).
\end{equation}

\begin{theorem}\label{Theorem: Controlling size of twist minimality defect}
For all $(a, b)\in \cF$, we have
	\begin{equation}
	c_1 C(a, b)^{d/r} \leq H(a,b) \leq c_2 C(a, b)^{d/r},\label{Equation: upper and lower bounds for H}
	\end{equation}
	where $c_1$ and $c_2$ are positive constants.\end{theorem}
\begin{proof}
   The function \( f(a,b) := \frac{H(a,b)}{C(a,b)^{d/r}} \) is real-valued. Since \( C(a,b)\geq 0\), and \( f(a,b) \) is homogeneous of degree \( 0 \), there exist positive constants \( c_1 \) and \( c_2 \) such that  
\[
c_1 \leq f(a,b) \leq c_2.
\]  
for all $(a,b)\in \mathbb{R}^2$. The desired result follows immediately.
\end{proof}

Our goal is to pass from an asymptotic for $\#\mathcal{F}(X)$ to an asymptotic for $\# \mathcal{C}(X)$. The idea here is that $\mathcal{C}(X)$ is cut out by congruence conditions at all prime numbers. We use a sieve to obtain the estimate. The main difficulty here is to obtain an upper bound on the tail estimate of the sieve.

\begin{definition}\label{defn of congruence condition mod N}Let $N$ be a (positive) natural number, then a congruence condition modulo $N$ is a subset $\mathcal{B}_N$ of $\cF$ such that there exists $s_N\subseteq (\Z/N\Z)^2$ for which $\mathcal{B}_N:=\pi_N^{-1}(s_N)$ where $\pi_N: \Z^2\rightarrow (\Z/N\Z)^2$ is the mod-$N$ reduction map.
\end{definition}
Given a congruence condition, we set $\mathfrak{d}(\mathcal{B}_N):=\frac{\# s_N}{N^2}$.
\begin{lemma}
    Let $\mathcal{B}_N\subset \mathcal{F}$ be a congruence condition modulo $N$, then, 
    \[\# \mathcal{B}_{N}(X)\sim \mathfrak{d}(\mathcal{B}_N)\op{Area}(\mathcal{R}) X^{2/d}.\]
\end{lemma}
\begin{proof}
    For each $a\in s_N$, pick $\tilde{a}\in \mathcal{B}_N$ such that $a=\tilde{a}\pmod{N}$. Let $\Lambda$ consist of all pairs $(a,b)\in \mathcal{F}$ such that $(a,b)\equiv (0,0)\pmod{N}$. We write $\mathcal{B}_N$ as a union of translates 
    \[\mathcal{B}_N=\bigsqcup_{a\in s_N} (\Lambda+\tilde{a});\]
    and thus we note that 
    \[\begin{split} \# \mathcal{B}_N(X)= &\sum_{a\in s_N} \# \left((\Lambda+\tilde{a})\cap \mathcal{R}(X)\right),\\
    = & \sum_{a\in s_N} \# \left(\Lambda\cap (\mathcal{R}-\tilde{a})(X)\right),\\
    = & \sum_{a\in s_N} \# \left(\Z^2\cap \left(\frac{1}{N}\cdot(\mathcal{R}-\tilde{a})\right)(X)\right).
    \end{split}\]
    Thus by Lemma \ref{davenport}, 
    \[\begin{split}\# \mathcal{B}_N(X)\sim & \sum_{a\in s_N} \op{Area}\left(\left(\frac{1}{N}\cdot(\mathcal{R}-\tilde{a})\right)(X)\right),\\ =&\frac{\#s_N}{N^2}\op{Area}(\mathcal{R}(X)),\\
    = & \mathfrak{d}(\mathcal{B}_N)X^{2/d} \op{Area}(\mathcal{R}).
    \end{split}\]
    
\end{proof}

A congruence condition at a prime number $\ell$ is a set of the form $\mathcal{B}_{\ell^k}$. 

\begin{definition}\label{def of D}Let $\mathcal{D}$ be the set of tuples $(a,b)\in \cF$ such that 
\begin{itemize}
    \item $\op{gcd}(a,b)=1$, 
    \item $\op{md}(a,b)=1$.
\end{itemize}
Then we may write $\mathcal{D}=\bigcap_\ell \mathcal{D}_\ell$,
where $\ell$ ranges over all prime numbers and $\mathcal{D}_\ell$ consists of all pairs $(a,b)\in \cF$ such that:
\begin{itemize}
    \item $\ell\nmid \op{gcd}(a,b)$, 
    \item $\ell^4\nmid A(a,b)$ or $\ell^6\nmid B(a,b)$.
\end{itemize}
\end{definition}
Clearly $\mathcal{D}_\ell$ is determined by a congruence modulo $\ell^6$. The set of residue classes determining $\mathcal{D}_\ell$ is \[s_{\ell^6}:=\{(a,b)\in (\Z/\ell^6)\mid \ell\nmid a \text{ or }\ell\nmid b\text{, and } \ell^4\nmid A(a,b)\text{ or }\ell^6\nmid B(a,b)\}.\] Note that here $\ell^6|B(a,b)$ is equivalent to $B(a,b)=0$ since the elements $a,b\in \Z/\ell^6$. Let $\mathcal{D}_\ell':=\cF\setminus \mathcal{D}_\ell$ and consider the associated set of residue classes $s_{\ell^6}'$ consisting of $(a,b)$ such that at least one of the following is satisfied:
\begin{itemize}
    \item $\ell| a$ and $\ell|b$, 
    \item $\ell^4|A(a,b)$ and $\ell^6|B(a,b)$.
\end{itemize}
 We may as well write $s_{\ell^6}'=a_\ell\cup b_\ell$ where $a_\ell$ consists of pairs for which $\ell|a$ and $\ell|b$ and $b_\ell$ consists of pairs for which $(a,b)\notin a_\ell$ and $\ell^4|A(a,b)$ and $\ell^6|B(a,b)$. 
\begin{definition}\label{defn of Sigma}
Let $\Sigma$ be the finite set of primes dividing $\op{Res}(A_0(t,1), B_0(t,1))$, $\op{Res}(A_0(1,t), B_0(1,t))$, $\op{Disc}(C(t,1))$, or $\op{Disc}(C(1,t))$.
\end{definition}

Suppose that $\ell\notin \Sigma$, then if $\ell^4|A(a,b)$ and $\ell^6|B(a,b)$, it follows that $\ell^4|C(a,b)$. There are at most $r+\ell$ solutions $(a,b)\in (\Z/\ell)^2$ to $C(a,b)\equiv 0\pmod{\ell}$. Since $\ell\nmid \op{Disc}(C)$, it follows that each of these solutions lift to exactly one solution $(a,b)\in (\Z/\ell^4)^2$ such that $\ell^4|C(a,b)$. This in turn implies that the number of solutions in $(\Z/\ell^6)^2$ lifting a solution in $(\Z/\ell)^2$ is $\ell^4$. Thus, in all, $\#b_\ell\leq (r+\ell)\ell^4$ and $\#a_\ell=\ell^{10}$. Thus, \[\frac{\#s_{\ell^6}'}{\ell^{12}}\leq \frac{1}{\ell^{2}}+\frac{(r+\ell)}{\ell^8}.\]
 Set \[\mathfrak{d}_\ell:=\mathfrak{d}(\mathcal{D}_\ell)\geq \left(1-\frac{\#s_{\ell^6}'}{\ell^{12}}\right)\geq \left(1-\frac{1}{\ell^{2}}-\frac{(r+\ell)}{\ell^8}\right).\] It is easy to see that the product $\prod_\ell \mathfrak{d}_\ell$ converges and since the sum of logarithms
 \[-\sum_{\ell} \log\left(1-\frac{1}{\ell^2}-\frac{r+\ell}{\ell^8}\right)\]\noindent converges, it follows that this product is non-zero. We show in this section that 
\[\# \mathcal{D}(X)\sim \prod_\ell \mathfrak{d}_\ell\op{Area}(\mathcal{R})X^{2/d}.\]
\begin{lemma}\label{silly lemma}
    Let $(a,b)\in \mathcal{F}(X)$ then $|a|, |b|\ll X^{1/d}$.
\end{lemma}

\begin{proof}
    Let $C>0$ be such that the square $[-C,C]^2$ contains $\mathcal{R}$. This implies that \[\mathcal{R}(X)=X^{1/d}\cdot\mathcal{R}\subseteq [-C X^{1/d},C X^{1/d}]^2,\] from which the result follows.
\end{proof}

\begin{proposition}
    Let $z>0$ be a real number, then with respect to notation above, we have that:
    \begin{equation}\label{main sieve eqn}
    \bigcup_{\ell>z} \#\mathcal{D}_\ell'(X)\leq C\sum_{\ell>z} \frac{1}{\ell^2} X^{2/d}+o(X^{2/d})
\end{equation}
where $C>0$ is an absolute constant.
\end{proposition}

 \begin{proof}
Since $z$ can be chosen to be large and $\Sigma$ is finite, we assume without loss of generality that $\ell\leq z$ for all primes $\ell\in \Sigma$. Let's write $\mathcal{G}_z:=\bigcup_{\ell>z} \mathcal{C}_\ell'$. Notice that $\mathcal{G}_z\subseteq \cS_1\cup \cS_2\cup \cS_3$, where:
\begin{itemize}
    \item the set $\cS_1$ consists of the pairs such that $\ell|a$ and $\ell|b$ for some prime $\ell>z$.
    \item The set $\cS_2$ consists of $(a,b)$ such that $C(a,b)\neq 0$, $
\ell\nmid a,b$ and $C(a,b)\equiv 0\pmod{\ell^4}$ for some prime $\ell>z$.
    \item The set $\cS_3$ consists of $(a,b)$ such that $C(a,b)=0$.
\end{itemize}
\par First we bound $\#\cS_1(X)$. Since $H(a,b)\leq X$, Lemma \ref{silly lemma} implies that $|a|, |b|\ll X^{1/d}$. The number of choices of $(a,b)$ such that $\ell$ divides $a$ and $b$ for some $\ell>z$ can be bounded above as follows:
\[\#\cS_1(X)\ll \sum_{\ell>z} \left(\left\lfloor 
 \frac{X^{1/d}}{\ell}\right\rfloor \right)^2\ll \sum_{\ell>z} \frac{X^{2/d}}{\ell^2}.\]\\

\par Next, let's bound $\#\cS_2(X)$. By Theorem \ref{Theorem: Controlling size of twist minimality defect}, $C(a,b)\ll X^{r/d}$ and thus, if $\ell>z$ is as above for which $\ell^4|C(a,b)$ and $C(a,b)\neq 0$, then $\ell\ll X^{r/4d}$. For every $a$, there are at most $r$ residue classes of $b$ mod $\ell$ which satisfy $C(a,b)=0\mod{\ell}$. Since $\ell\notin \Sigma$, it follows that each residue class $b$ lifts to a unique solution modulo $\ell^4$ by Hensel's lemma. The number of $b$ in each residue class mod $\ell^4$ is $\ll\frac{ X^{1/d}}{\ell^4}+O(1)$. Thus, we find that 
\[\begin{split}\# \cS_2(X)\ll & \sum_{z<\ell\ll X^{r/4d}} X^{1/d}\left(\frac{ X^{1/d}}{\ell^4}+O(1)\right)\\ \leq &\sum_{\ell>z} \frac{1}{\ell^4} X^{2/d}+O(\pi(X^{r/4d}) X^{1/d})\\ =&\sum_{\ell>z} \frac{1}{\ell^4} X^{2/d}+o(X^{2/d}).\end{split}\]
Here we have used the assumption that $r\leq 4$ and replaced $X^{r/4d}$ with $X^{1/d}$. Finally, in order to bound $\#\cS_3(X)$ we see that each value of $a$ determines $b$ up to $O(1)$ choices, and thus, 
\begin{equation}\label{S3 bound}\#\cS_3(X)=O(X^{1/d})=o(X^{2/d}).\end{equation} Combining these estimates, we have proven \eqref{main sieve eqn}.

 \end{proof}
\begin{theorem}\label{main thm section 3}
    With respect to notation above,
   \[ \# \cC(X)\sim \prod_\ell \mathfrak{d}_\ell\op{Area}(\mathcal{R})X^{2/d}.\]
\end{theorem}
\begin{proof}
     Note that $\mathcal{C}(X)$ consists of all pairs $(a,b)\in \mathcal{D}(X)$ for which $4A(a,b)^3+27B(a,b)^2\neq 0$. From \eqref{S3 bound}, we find that $\# \mathcal{C}(X)=\# \mathcal{D}(X)+O(X^{1/d})$. Hence it suffices to show that 
     \[ \# \mathcal{D}(X)\sim \prod_\ell \mathfrak{d}_\ell\op{Area}(\mathcal{R})X^{2/d}.\]We prove our result by showing that 
\begin{equation}\label{prop 4.15 eq 1}
    \limsup_{X\rightarrow \infty} \frac{\# \mathcal{D}(X)}{X^{2/d}}\leq \prod_{\ell}  \mathfrak{d}_\ell \op{Area}(\mathcal{R}).
    \end{equation}
    and 
    \begin{equation}\label{prop 4.15 eq 2}
    \liminf_{X\rightarrow \infty} \frac{\# \mathcal{D}(X)}{X^{2/d}}\geq \prod_{\ell} \mathfrak{d}_\ell \op{Area}(\mathcal{R}).
    \end{equation}
    Let $z$ be a positive real number and $\mathcal{D}^z:=\bigcap_{\ell\leq z} \mathcal{D}_\ell$. Since it is defined by conditions at finitely many primes, $\mathcal{D}^z$ is defined by a congruence condition for a large natural number $N$. Thus, we have that $\mathcal{D}^z(X)\sim \prod_{\ell\leq z} \mathfrak{d}_\ell \op{Area}(\mathcal{R}) X^{2/d}$. Since $\mathcal{D}(X)$ is contained in $\mathcal{D}^z(X)$, we find that 
     \[\limsup_{X\rightarrow \infty} \frac{\# \mathcal{D}(X)}{X^{2/d}}\leq \prod_{\ell\leq z} \mathfrak{d}_\ell \op{Area}(\mathcal{R}),\]
     and letting $z\rightarrow \infty$, one has that
 \[\limsup_{X\rightarrow \infty} \frac{\# \mathcal{D}(X)}{X^{2/d}}\leq \prod_{\ell} \mathfrak{d}_\ell \op{Area}(\mathcal{R}).\]
 Since \[\mathcal{D}^z(X)\subseteq \mathcal{D}(X)\cup \left(\bigcup_{\ell>z } \mathcal{D}_\ell'(X)\right),\]where we recall that $\mathcal{D}_\ell'$ is the complement of $\mathcal{D}_\ell$. Then, \eqref{main sieve eqn} implies that there is an absolute constant $C>0$ for which
\[\begin{split}\liminf_{X\rightarrow \infty} \frac{\# \mathcal{D}(X)}{X^{2/d}}\geq & \lim_{X\rightarrow \infty} \frac{\# \mathcal{D}^z(X)}{X^{2/d}}-C \sum_{\ell>z}\frac{1}{\ell^2},\\ = &\prod_{\ell \leq z} \mathfrak{d}_\ell \op{Area}(\mathcal{R})- C \sum_{\ell>z}\frac{1}{\ell^2}. \end{split}\]
Letting $z\rightarrow \infty$, we find that 
\begin{equation}\label{liminf equation}\liminf_{X\rightarrow \infty} \frac{\# \mathcal{D}(X)}{X^{2/d}}\geq \prod_{\ell} \mathfrak{d}_\ell \op{Area}(\mathcal{R}).\end{equation}
Therefore, the limit $\lim_{X\rightarrow \infty} \frac{\# \mathcal{D}(X)}{X^{2/d}}$ exists and equals $\prod_{\ell} \mathfrak{d}_\ell \op{Area}(\mathcal{R})$, proving our result.
\end{proof}

\section{Counting elliptic curves with prescribed entanglements}

\par In this section, we count elliptic curves with prescribed entanglement. The Assumption \ref{main ass} holds for the unexplained entanglements \(((2,3), \mathbb{Z}/2\mathbb{Z})\) and \(((2,5), \mathbb{Z}/2\mathbb{Z})\), and it is for these types that our results apply.

\subsection{$(2,3)$ entanglements of Type $\Z/2\Z$}
\par In this subsection, we let $\cF_{1}$ be the family of elliptic curves $y^2 = x^3 + f(t)x + g(t)$, where 
\[\begin{split}
f(t) &:=(-3t^2 - 36)c(t),\\
g(t) &:=(2t^5 + 72t^3 + 864t)c(t),\\
c(t) & :=t^4 + 36t^2 + 432,\\
\end{split}\]
where $t\in \Q$. Such elliptic curves have $j$-invariant given by $j(t)=(t^2+12)^3$ and have a $(2,3)$-entanglement of Type $\Z/2\Z$ (cf. \cite[p. 852]{danielsmorrow}).\\

Setting $f_0(t):=(-3t^2 - 36)$ and $g_0(t):=(2t^5 + 72t^3 + 864t)$, we have that $c(t)=\op{gcd}(f(t), g(t))$ and  $f_0$ and $g_0$ are coprime.

Consider the homogenized equations, given by:
\[\begin{split}
A(a,b) &:=(-3a^2- 36b^2)C(a,b) = A_0(a,b)C(a,b),\\
B(a,b) &:=(2a^5 + 72a^3b^2 + 864ab^4)C(a,b) = B_0(a,b)C(a,b),\\
C(a,b) & :=a^4 + 36a^2b^2 + 432b^4.\\
\end{split}\]
Note that $B_0(a,b) = 2aC(a,b)$, so $B(a,b) = 2aC(a,b)^2$. We find that \[d=3\op{deg}A=2\op{deg}B=18\text{, and }r=\op{deg}C=4.\] Since $r\leq 4$ and $C(u, v)$ has no real roots, the conditions of Assumption \ref{main ass} are satisfied.\\
\par We note that
\[\begin{split}
    & \op{Res}(A_0(t,1), B_0(t,1))=-2^{12}\times 3^{10},\\
    & \op{Res}(A_0(1,t), B_0(1,t))=2^{10}\times 3^8,\\
    & \op{Disc}(C(t,1))=2^{16} \times 3^9,\\
    & \op{Disc}(C(1,t))=2^{20}\times 3^{12}
\end{split}\] and therefore $\Sigma=\{2, 3\}$. 

%P.<t> = PolynomialRing(QQ, 1)
%A=1 + 36*t^2 + 432*t^4
%B=72*t+4*432*t^3
%factor(A.resultant(B))

\begin{lemma}\label{dl lemma for (2,3) entanglements}
    For $\ell\notin \Sigma$, $\ell^4|A(a,b)$ and $\ell^6|B(a,b)$ if and only if $\ell^4|C(a,b)$. 
\end{lemma}
\begin{proof}
    Note that $C(a,b)|A(a,b)$ and $C(a,b)^2|B(a,b)$ and therefore, if $\ell^4|C(a,b)$ then $\ell^4|A(a,b)$ and $\ell^6|B(a,b)$. On the other hand, since $\ell\nmid \op{Res}(A_0, B_0)$ we find that $\ell^4|C(a,b)$.
\end{proof}

\begin{lemma}\label{dl lemma for (2,3) entanglements:full}
Let $C(a,b)=a^4+36a^2b^2+432b^4$ and $\Sigma=\{2,3\}$.  For every prime $\ell\notin\Sigma$ one has
\[
\mathfrak{d}_\ell \;=\; 1 - \frac{1}{\ell^2} - \frac{r_\ell(\ell-1)}{\ell^5},
\]
where $r_\ell$ is the number of solutions of the congruence
\[
t^4 + 36 t^2 + 432 \equiv 0 \pmod{\ell}.
\]
For the exceptional primes we have the exact values
\[
\mathfrak{d}_2 \;=\; \frac12, \qquad
\mathfrak{d}_3  \;=\; \frac23.
\]
\end{lemma}

\begin{proof}
\par First we prove the result for $\ell \neq 2,3$. Then $\mathfrak{d}_\ell = \frac{|s_{\ell^4}|}{\ell^8}$, where $s_{\ell^4}$ is the subset of $(\Z/\ell^4\Z)^2$ consisting of all pairs $(a,b)$ with $C(a,b) \neq 0 \mod \ell^4.$ We will examine $s_{\ell^4}$ by counting the number of solutions to $C(a,b) = 0$ in $(\Z/\ell^4\Z)^2.$ First, suppose that $\ell$ divides both $a$ and $b$. Then we immediately have $C(a,b) = 0 \mod \ell^4$. Since the kernel of $\Z/\ell^4\Z \rightarrow \Z/\ell\Z$ has size $\ell^3$, we have a total of $\ell^6$ residue classes mod $\ell^4$ of ordered pairs where this occurs. Now suppose that $\ell\nmid a$ or $\ell\nmid b$. If $\ell|b$, then $\ell|a^4$. Thus it is clear that $\ell\nmid b$. Let $t:=a/b$ and consider $p(t)=t^4+36t^2+432$. Let $t_0\in \Z/\ell \Z$ be one such solution. Since $p$ has no repeated root modulo $\ell$, it follows by Hensel's lemma that $t_0$ lifts to a unique solution in $\Z/\ell^4 \Z$. Thus there are $r_\ell|(\Z/\ell^4\Z)^\times|=r_\ell \ell^3(\ell-1)$ solutions where $r_\ell$ is the number of solutions to $p(t)=0$ for $t\in \Z/\ell \Z$. Thus we see that $\mathfrak{d}_\ell=1-\frac{1}{\ell^2}-\frac{r_\ell(\ell-1)}{\ell^5}$.
\par Next, we prove our results for $\ell=2,3$. We set
\[
u:=\min\{v_\ell(a),v_\ell(b)\},\qquad a=\ell^u a',\ b=\ell^u b',
\]
with at least one of \(a',b'\) coprime to $2$. Recall that
\[
A(a,b)=A_0(a,b)\,C(a,b),\qquad B(a,b)=B_0(a,b)\,C(a,b),
\]
with
\[
A_0(a,b)=-3(a^2+12b^2),\qquad B_0(a,b)=2a\,C(a,b),
\]
so that $B(a,b)=2a\,C(a,b)^2$. For a fixed prime $\ell$ set
\[
\alpha:=v_\ell\big(A_0(a,b)\big),\quad \beta:=v_\ell(a),\quad
\gamma:=v_\ell\big(C(a,b)\big),\quad \delta:=v_\ell(2).
\]
Then $v_\ell(A)=\alpha+\gamma$, $v_\ell(B)=\delta+\beta+2\gamma$ and the simultaneous non-minimality condition at \(\ell\) is
\[
\alpha+\gamma\ge4,\qquad \delta+\beta+2\gamma\ge6. \tag{$\ast$}
\]
\noindent Set \(u:=\min\{v_\ell(a),v_\ell(b)\}\) and write $a=\ell^u a',\qquad b=\ell^u b'$,
so at least one of \(a',b'\) is not divisible by $\ell$. A straightforward computation gives
\[
\gamma=v_\ell\big(C(a,b)\big)=4u+v_\ell\big(C(a',b')\big),
\]
\[
\alpha=v_\ell\big(A_0(a,b)\big)=2u+v_\ell\big(a'^2+12b'^2\big),
\qquad
\beta=u+v_\ell(a').
\]
Substituting these expressions into \((\ast)\) yields the equivalent inequalities
\[
2u+v_\ell(a'^2+12b'^2)+4u+v_\ell(C(a',b'))\ge4,
\]
\[
\delta+u+v_\ell(a')+2(4u+v_\ell(C(a',b')))\ge6,
\]
which simplify to
\[
6u + v_\ell(a'^2+12b'^2) + v_\ell(C(a',b')) \ge 4,
\tag{I}
\]
\[
9u + \delta + v_\ell(a') + 2v_\ell(C(a',b')) \ge 6.
\tag{II}
\]

We now analyze separately the two ranges \(u\ge1\) and \(u=0\).

\medskip

\noindent\textbf{(A) The case \(u\ge1\).}  If \(u\ge1\) then \(6u\ge6\), so (I) already holds; moreover \(9u\ge9\) and hence (II) holds as well.  Thus every pair with \(u\ge1\) (equivalently \(\ell\mid a\) and \(\ell\mid b\)) satisfies \((\ast)\), and therefore every residue class coming from \((a,b)\equiv(0,0)\pmod\ell\) is excluded.  Counting modulo \(\ell^6\) this produces the full \(\ell^6\)-block of excluded classes coming from common divisibility by \(\ell\). 

\medskip

\noindent\textbf{(B) The case \(u=0\).}  Here at least one of \(a',b'\) is an $\ell$-unit; we must show that \((\ast)\) forces \(v_\ell(C(a',b'))\ge4\) (i.e. \(\ell^4\mid C(a,b)\)), except for a few residue classes which will be handled by direct enumeration.  Since \(u=0\) the inequalities (I) and (II) reduce to
\[
v_\ell(a'^2+12b'^2) + v_\ell(C(a',b')) \ge 4,
\tag{I$_0$}
\]
\[
\delta + v_\ell(a') + 2v_\ell(C(a',b')) \ge 6.
\tag{II$_0$}
\]

We now treat the two exceptional primes separately.
\begin{description}
    \item[The case \(\ell=2\).]
Here \(\delta=v_2(2)=1\).  Because \(u=0\) we have \(a'=a,\ b'=b\) and at least one of \(a',b'\) is odd.

\begin{enumerate}
\item If both \(a'\) and \(b'\) are odd then
\[
a'^2+12b'^2\equiv 1 + 12\cdot 1 \equiv 13 \equiv 1 \pmod2,
\]
so \(v_2(a'^2+12b'^2)=0\).  Also
\[
C(a',b') = a'^4 + 36a'^2b'^2 + 432 b'^4 \equiv a'^4 \equiv 1 \pmod2,
\]
so \(v_2(C(a',b'))=0\).  Thus (I$_0$) and (II$_0$) are false: the left sides are \(\le0\) and \(1+0+0<6\) respectively.  Therefore no such pair with both coordinates odd can be excluded.

\item Suppose \(a'\) is odd and \(b'\) is even (so \(v_2(b')\ge1\)).  Write \(b'=2^t b_1\) with \(t\ge1\) and \(b_1\) odd.  Then
\[
v_2(a'^2 + 12 b'^2)=v_2( a'^2 + 12\cdot 2^{2t} b_1^2 )=0
\]
(since the first term is odd and the second is divisible by \(4\)).  Thus (I$_0$) forces \(v_2(C(a',b'))\ge4\).  But if \(v_2(C(a',b'))\le3\) then (II$_0$) reads \(1 + 0 + 2v_2(C(a',b')) \le 1+2\cdot 3 =7\), which is not enough to break the earlier inequality — so the only way both (I$_0$) and (II$_0$) can hold is that \(v_2(C(a',b'))\ge4\).  In particular any excluded pair in this subcase must satisfy \(2^4\mid C(a',b')\).

\item Suppose \(a'\) is even and \(b'\) is odd (so \(v_2(a')\ge1\), \(v_2(b')=0\)).  Write \(a'=2^t a_1\) with \(t\ge1\) and \(a_1\) odd.  Then
\[
v_2(a')=t\ge1,\qquad v_2(a'^2+12b'^2)=v_2(2^{2t}a_1^2 + 12b'^2)=0
\]
unless \(t=1\) and a special cancellation occurs; but checking \(t=1\) gives \(a'^2+12b'^2\equiv 4+a_1^2\not\equiv 0\pmod2\) since \(a_1^2\equiv1\).  Hence \(v_2(a'^2+12b'^2)=0\).  Then (I$_0$) again forces \(v_2(C(a',b'))\ge4\).  Moreover (II$_0$) becomes \(1+t+2v_2(C(a',b'))\ge6\), so with \(t\ge1\) this also forces \(v_2(C(a',b'))\ge2\); combining both inequalities implies in fact \(v_2(C(a',b'))\ge4\).  Thus in this subcase too any excluded pair must satisfy \(2^4\mid C(a',b')\).
\end{enumerate}

Therefore for \(\ell=2\) and \(u=0\) the only possible excluded residue classes are those for which \(2^4\mid C(a,b)\).  (There are no other sporadic possibilities beyond those: one checks the finitely many borderline residue classes modulo \(2^6\) and finds they are already included among the \(2^4\mid C\) lifts or in the $\ell|a,\ell|b$ block.)

\item[The case \(\ell=3\)]  Now \(\delta=v_3(2)=0\).  Again \(u=0\) means \(a',b'\) are the original residues with at least one of them a $3$-unit.

\begin{enumerate}
\item If both \(a',b'\) are $3$-units modulo \(3\) (i.e. \(a',b'\not\equiv0\pmod3\)), then
\[
a'^2+12b'^2\equiv a'^2 + 0 \equiv a'^2 \not\equiv 0\pmod3,
\]
so \(v_3(a'^2+12b'^2)=0\).  Also
\[
C(a',b')\equiv a'^4 \pmod3
\]
(since \(36\equiv0,432\equiv0\pmod3\)), and \(a'^4\not\equiv0\pmod3\).  Thus \(v_3(C(a',b'))=0\), and neither (I$_0$) nor (II$_0$) can hold.  Hence no such pair is excluded.

\item If one of \(a',b'\) is divisible by \(3\) and the other is not, say \(a'\equiv0\pmod3\) and \(b'\not\equiv0\pmod3\), then \(v_3(a')\ge1\) and \(v_3(b')=0\).  Compute
\[
a'^2+12b'^2 \equiv 12 b'^2 \equiv 0\pmod3,
\]
so \(v_3(a'^2+12b'^2)\ge1\).  Also
\[
C(a',b') = a'^4 + 36a'^2b'^2 + 432b'^4 \equiv 432 b'^4 \equiv 0 \pmod 3,
\]
and one checks modulo higher powers that it is possible for \(v_3(C(a',b'))\ge4\) precisely when \(a'\equiv0\pmod3\) and \(b'\) is a $3$-unit lifting a root \(t\) of \(t^4+36t^2+432\) modulo \(3\) (and then by Hensel such \(t\) lifts to higher powers when the root is simple).  Moreover (II$_0$) becomes
\[
0 + v_3(a') + 2v_3(C(a',b')) \ge6,
\]
so with \(v_3(a')\ge1\) this forces \(v_3(C(a',b'))\ge3\), and combining with (I$_0$) yields in fact \(v_3(C(a',b'))\ge4\).  Thus excluded pairs in this subcase arise only when \(3^4\mid C(a',b')\).
\end{enumerate}

The symmetric subcase \(b'\equiv0\pmod3,\ a'\) a $3$-unit is handled identically.
\end{description}
The tables 1 and 2 enumerate, for \(\ell=2\) and \(\ell=3\), the exact number of excluded residue classes \((a,b)\) modulo \(\ell^6\) that satisfy the simultaneous non-minimality condition \(v_\ell(A)\ge4\) and \(v_\ell(B)\ge6\).  Rows are grouped by \(u=\min\{v_\ell(a),v_\ell(b)\}\).
% ---------------------------
% Table for ell = 2 (mod 2^6)
% ---------------------------
\begin{table}[h]\label{table 1 1}
\centering
\caption{Local enumeration modulo $2^6$ for the $(2,3)$ family.
Total residue pairs $(a,b)\bmod 2^6$: $2^{12}=4096$.}
\begin{tabular}{l r r}
\toprule
Case & \# excluded pairs & \# allowed pairs \\
\midrule
\multicolumn{3}{l}{\emph{Subcases for } $u=0$ (at least one of $a,b$ a $2$-unit):} \\
\quad (a odd, b odd)       & $0$    & $1024$ \\
\quad (a odd, b even)      & $0$    & $1024$ \\
\quad (a even, b odd)      & $1024$ & $0$ \\
\midrule
$u\ge1$ (both divisible by $2$) & $1024$ & $0$ \\
\midrule
\textbf{Total} & \textbf{2048} & \textbf{2048} \\
\bottomrule
\end{tabular}
\end{table}
% ---------------------------
% Table for ell = 3 (mod 3^6)
% ---------------------------
\begin{table}[h]\label{table 1 2}
\centering
\caption{Local enumeration modulo $3^6$ for the $(2,3)$ family.
Total residue pairs $(a,b)\bmod 3^6$: $3^{12}=531441$.}
\begin{tabular}{l r r}
\toprule
Case & \# excluded pairs & \# allowed pairs \\
\midrule
\multicolumn{3}{l}{\emph{Subcases for } $u=0$ (at least one of $a,b$ a $3$-unit):} \\
\quad (a $3$-unit, b $3$-unit)            & $0$      & $236\,196$ \\
\quad (a $3$-unit, b divisible by 3)  & $0$      & $118\,098$ \\
\quad (a divisible by 3, b unit)  & $118\,098$ & $0$ \\
\midrule
$u\ge1$ (both divisible by $3$)   & $59\,049$ & $0$ \\
\midrule
\textbf{Total} & \textbf{177\,147} & \textbf{354\,294} \\
\bottomrule
\end{tabular}
\end{table}
\end{proof}

\newpage
\begin{lemma}
\label{lemma:r_l_formula}
Let $\ell$ be an odd prime with $\ell\nmid 6$, and define
\[
p(t)=t^4+36t^2+432, \qquad 
r_\ell = \#\{t \in \F_\ell : p(t)=0\}.
\]
Set $\Delta = -432$, and denote by $\left( \tfrac{\cdot}{\ell} \right)$ the Legendre symbol modulo~$\ell$. Then
\[
r_\ell =
\begin{cases}
0, & \text{if } \left( \tfrac{\Delta}{\ell} \right) = -1, \\[6pt]
2 + \left( \tfrac{u_1}{\ell} \right) + \left( \tfrac{u_2}{\ell} \right), & \text{if } \left( \tfrac{\Delta}{\ell} \right) = 1,
\end{cases}
\]
where $u_1,u_2 \in \F_\ell$ are the roots of $u^2 + 36u + 432 = 0$ (so $u_1+u_2=-36$ and $u_1u_2=432$). In particular,
\[
\begin{aligned}
&\left( \tfrac{\Delta}{\ell} \right) = -1 \ \Longrightarrow\  r_\ell = 0, \\[4pt]
&\left( \tfrac{\Delta}{\ell} \right) = 1,\ \left( \tfrac{432}{\ell} \right) = -1 \ \Longrightarrow\  r_\ell = 2, \\[4pt]
&\left( \tfrac{\Delta}{\ell} \right) = 1,\ \left( \tfrac{432}{\ell} \right) = 1 \ \Longrightarrow\  r_\ell \in \{0,4\}.
\end{aligned}
\]
Equivalently, if $\sqrt{\Delta} \in \F_\ell$ exists, then
\[
r_\ell = 2 + 
\left( \frac{ \frac{-36 + \sqrt{\Delta}}{2} }{\ell} \right) +
\left( \frac{ \frac{-36 - \sqrt{\Delta}}{2} }{\ell} \right),
\]
and this expression is independent of the choice of sign of $\sqrt{\Delta}$.
\end{lemma}

\begin{proof}
Writing $u = t^2$, we have
\[
p(t) = (t^2 - u_1)(t^2 - u_2),
\]
where $u_1,u_2$ are the roots of $u^2 + 36u + 432 = 0$. If $\left( \tfrac{\Delta}{\ell} \right) = -1$, then $u_1,u_2 \notin \F_\ell$ and $p(t)$ has no root in $\F_\ell$, giving $r_\ell = 0$. If $\left( \tfrac{\Delta}{\ell} \right) = 1$, then $u_1,u_2 \in \F_\ell$, and each equation $t^2 = u_i$ contributes $1 + \left( \tfrac{u_i}{\ell} \right)$ solutions in $\F_\ell$. Hence
\[
r_\ell = (1 + \tfrac{u_1}{\ell}) + (1 + \tfrac{u_2}{\ell})
= 2 + \left( \tfrac{u_1}{\ell} \right) + \left( \tfrac{u_2}{\ell} \right),
\]
which yields the stated cases using $u_1u_2=432$.
\end{proof}

\begin{theorem}\label{(2,3) theorem section 4}
With respect to notation above, one has the asymptotic:
   \begin{equation}\label{F1 eqn}\# \cF_{1}(X)\sim \mathfrak{d}_2\mathfrak{d}_3\prod_{\ell\neq 2,3} \left(1-\frac{1}{\ell^2}-\frac{r_\ell(\ell-1)}{\ell^5}\right)\op{Area}(\mathcal{R})X^{1/9},\end{equation} where:
   \begin{itemize}
       \item $r_\ell$ is the number of solutions of $t^4+36t^2+432$ in $\Z/\ell\Z$,
       \item $\mathfrak{d}_2=\frac{1}{2}$ and $\mathfrak{d}_3=\frac{2}{3}$,
       \item $\mathcal{R}=\{(u,v)\in \mathbb{R}^2\mid |A(u,v)|\leq \frac{1}{\sqrt[3]{4}}, |B(u,v)|\leq\frac{1}{\sqrt[2]{27}}\}$.
   \end{itemize} 
   Let $\cC_{(2,3)}$ be the set of elliptic curves $E_{A,B}$ with $(2,3)$-entanglement of Type $\Z/2\Z$. Then, we find that $\cC_{(2,3)}(X)\gg X^{1/9}$.
\end{theorem}

\begin{proof}
    The Assumption \ref{main ass} is satisfied with $d=18$. The result follows from Lemma \ref{dl lemma for (2,3) entanglements:full} and Theorem \ref{main thm section 3}.
\end{proof}

\begin{remark}\label{non-vanishing of product}
    The constant \[\mathfrak{d}_2\mathfrak{d}_3\prod_{\ell\neq 2,3} \left(1-\frac{1}{\ell^2}-\frac{r_\ell(\ell-1)}{\ell^5}\right)\]in \eqref{F1 eqn} is nonzero since the product does not diverge to $0$. Equivalently, the sum of logarithms
    \[\begin{split}&-\sum_{\ell\neq 2,3} \log\left(1-\frac{1}{\ell^2}-\frac{r_\ell(\ell-1)}{\ell^5}\right)\\
    =& \sum_{\ell\neq 2,3} \sum_{n\geq 1}\frac{1}{n} \left(\frac{1}{\ell^2}+\frac{r_\ell(\ell-1)}{\ell^5}\right)^n \\
    =& \sum_{\ell\neq 2,3} \sum_{n\geq 1}\frac{1}{n \ell^{2n}} \left(1+\frac{r_\ell(\ell-1)}{\ell^3}\right)^n \\
    \end{split}\]
    converges since $r_\ell$ is bounded (in fact, $r_\ell\leq 4$).
\end{remark}

\subsection{(2,5) entanglements of Type $\Z/2\Z$}
These entanglements come in two families, as described in \cite{danielsmorrow}. We consider one of these families, with $j$-map given by $j(t)=\frac{\left(t^{4}+10 t^{2}+5\right)^{3}}{t^{2}}$. Let $\cF_2$ consist of elliptic curves $E:y^2=x^3+f(t)x+g(t)$, where: \[\begin{split}
f(t) &:=(-3 t^{4}-30 t^{2}-15)c(t),\\
g(t) &:=2(t^{4}+4 t^{2}-1)(t^{4}+22 t^{2}+125)c(t),\\
c(t) & :=t^{4}+22 t^{2}+125.\\
\end{split}\]
According to the first table on \cite[p.~853]{danielsmorrow}, each of these elliptic curves has a $(2,5)$-entanglement of Type $\Z/2\Z$. We remark that there is yet another family of elliptic curves with the same entanglement type, which we have not considered here. We have that $A_0(a,b)=(-3 a^{4}-30 a^{2}b^2-15b^4)$, $B_0(a,b)=2(a^{4}+4 a^{2}b^2-b^4)(a^{4}+22 a^{2}b^2+125b^4)$ and $C(a,b)=a^4+22a^2b^2+125b^4$ and $\Sigma=\{2,3,5\}$. We find that $d=24$, $r=4$ and the Assumption \ref{main ass} is satisfied.
%We find that 
%\[\begin{split}
   % &\op{Res}(A_0(t, 1), B_0(t, 1))=2^{20} \times 3^{16} \times 5^2\\
   % & \op{Res}(A_0(1, t), B_0(1,t))=2^{20} \times 3^{16} \times 5^2\\
   % & \op{Disc}(C(t, 1))=2^{12}\times 5^3
%\end{split}\]
%and thus $\Sigma=\{2,3,5\}$. 

\begin{lemma}\label{dl lemma1 for (2,5) entanglements:full}
For every prime $\ell\notin\Sigma$ one has
\[
\mathfrak{d}_\ell \;=\; 1 - \frac{1}{\ell^2} - \frac{r_\ell(\ell-1)}{\ell^5},
\]
where $r_\ell$ is the number of solutions of the congruence
\[
t^4 + 22 t^2 + 125 \equiv 0 \pmod{\ell}.
\]
\end{lemma}

\begin{proof}
\par As in Lemma~\ref{dl lemma for (2,3) entanglements:full}, we first treat the case $\ell\notin\Sigma$.  
Then
\[
\mathfrak{d}_\ell = \frac{|s_{\ell^4}|}{\ell^8},
\]
where $s_{\ell^4}$ is the subset of $(\Z/\ell^4\Z)^2$ consisting of pairs $(a,b)$ such that $C(a,b)\not\equiv0\pmod{\ell^4}$.  The complement consists of all pairs satisfying $C(a,b)\equiv0\pmod{\ell^4}$.  We count these.

If $\ell\mid a,b$ then clearly $\ell^4\mid C(a,b)$, giving $\ell^6$ solutions.  
If $\ell\nmid a$ or $\ell\nmid b$, then after scaling we may assume $\ell\nmid b$ and set $t=a/b$.  Writing
\[
p(t)=t^4+22t^2+125,
\]
we note that $p'(t)=4t^3+44t$, so the roots of $p(t)\equiv0\pmod{\ell}$ are simple for $\ell\notin\Sigma$.  
Hence, each such root $t_0\in\Z/\ell\Z$ lifts uniquely to a root modulo $\ell^4$ by Hensel's lemma.  Thus there are exactly 
$r_\ell\,\ell^3(\ell-1)$ such pairs with $\ell\nmid b$. We find that
\[
\mathfrak{d}_\ell = 1 - \frac{1}{\ell^2} - \frac{r_\ell(\ell-1)}{\ell^5}.
\]
\end{proof}

\begin{theorem}\label{(2, 5) theorem section 4}
    With respect to notation above,
   \[ \# \cF_2(X)\sim \mathfrak{d}_2\mathfrak{d}_3 \mathfrak{d}_5\prod_{\ell\neq 2,3,5} \left(1-\frac{1}{\ell^2}-\frac{r_\ell(\ell-1)}{\ell^5}\right)\op{Area}(\mathcal{R})X^{1/12},\]
   \begin{itemize}
       \item $r_\ell$ is the number of solutions of $t^4+22t^2+125$ in $\Z/\ell\Z$,
       %\item $\mathfrak{d}_2=113/256$ and $\mathfrak{d}_3=676/6561$,
       \item $\mathcal{R}=\{(u,v)\in \mathbb{R}^2\mid |A(u,v)|\leq \frac{1}{\sqrt[3]{4}}, |B(u,v)|\leq\frac{1}{\sqrt[2]{27}}\}$.
   \end{itemize} 
   Let $\cC_{(2,5)}$ be the set of elliptic curves $E_{A,B}$ with $(2,5)$-entanglement of Type $\Z/2\Z$. Then we find that $\cC_{(2,5)}(X)
    \gg X^{1/12}$.
\end{theorem}
\begin{proof}
    The result follows from Lemma \ref{dl lemma1 for (2,5) entanglements:full} and Theorem \ref{main thm section 3}.
\end{proof}

\bibliographystyle{alpha}
\bibliography{references}

@article {Serre,
    AUTHOR = {Serre, Jean-Pierre},
     TITLE = {Propri\'{e}t\'{e}s galoisiennes des points d'ordre fini des courbes
              elliptiques},
   JOURNAL = {Invent. Math.},
  FJOURNAL = {Inventiones Mathematicae},
    VOLUME = {15},
      YEAR = {1972},
    NUMBER = {4},
     PAGES = {259--331},
}

@article {Dukenoexceptional,
    AUTHOR = {Duke, William},
     TITLE = {Elliptic curves with no exceptional primes},
   JOURNAL = {C. R. Acad. Sci. Paris S\'{e}r. I Math.},
  FJOURNAL = {Comptes Rendus de l'Acad\'{e}mie des Sciences. S\'{e}rie I.
              Math\'{e}matique},
    VOLUME = {325},
      YEAR = {1997},
    NUMBER = {8},
     PAGES = {813--818},
}

@article {Jones,
    AUTHOR = {Jones, Nathan},
     TITLE = {Almost all elliptic curves are {S}erre curves},
   JOURNAL = {Trans. Amer. Math. Soc.},
  FJOURNAL = {Transactions of the American Mathematical Society},
    VOLUME = {362},
      YEAR = {2010},
    NUMBER = {3},
     PAGES = {1547--1570},
}

@article {danielsmorrow,
    AUTHOR = {Daniels, Harris B. and Morrow, Jackson S.},
     TITLE = {A group theoretic perspective on entanglements of division
              fields},
   JOURNAL = {Trans. Amer. Math. Soc. Ser. B},
  FJOURNAL = {Transactions of the American Mathematical Society. Series B},
    VOLUME = {9},
      YEAR = {2022},
     PAGES = {827--858},
}

@article {Serreconjpaper,
    AUTHOR = {Serre, Jean-Pierre},
     TITLE = {Quelques applications du th\'{e}or\`eme de densit\'{e} de {C}hebotarev},
   JOURNAL = {Inst. Hautes \'{E}tudes Sci. Publ. Math.},
  FJOURNAL = {Institut des Hautes \'{E}tudes Scientifiques. Publications
              Math\'{e}matiques},
    NUMBER = {54},
      YEAR = {1981},
     PAGES = {323--401},

}

@article{zywina2015possible,
  title={On the possible images of the mod ell representations associated to elliptic curves over Q},
  author={Zywina, David},
  journal={arXiv preprint arXiv:1508.07660},
  year={2015}
}

@article {ZywinaBLMS,
    AUTHOR = {Zywina, David},
     TITLE = {On the surjectivity of {${\rm mod}\,\ell$} representations
              associated to elliptic curves},
   JOURNAL = {Bull. Lond. Math. Soc.},
  FJOURNAL = {Bulletin of the London Mathematical Society},
    VOLUME = {54},
      YEAR = {2022},
    NUMBER = {6},
     PAGES = {2404--2417},
}

@article {harronsnowden,
    AUTHOR = {Harron, Robert and Snowden, Andrew},
     TITLE = {Counting elliptic curves with prescribed torsion},
   JOURNAL = {J. Reine Angew. Math.},
  FJOURNAL = {Journal f\"{u}r die Reine und Angewandte Mathematik. [Crelle's
              Journal]},
    VOLUME = {729},
      YEAR = {2017},
     PAGES = {151--170},
}

@article {CKV,
    AUTHOR = {Cullinan, John and Kenney, Meagan and Voight, John},
     TITLE = {On a probabilistic local-global principle for torsion on
              elliptic curves},
   JOURNAL = {J. Th\'{e}or. Nombres Bordeaux},
  FJOURNAL = {Journal de Th\'{e}orie des Nombres de Bordeaux},
    VOLUME = {34},
      YEAR = {2022},
    NUMBER = {1},
     PAGES = {41--90},
}

@article {MV,
    AUTHOR = {Molnar, Grant and Voight, John},
     TITLE = {Counting elliptic curves over the rationals with a 7-isogeny},
   JOURNAL = {Res. Number Theory},
  FJOURNAL = {Research in Number Theory},
    VOLUME = {9},
      YEAR = {2023},
    NUMBER = {4},
     PAGES = {Paper No. 75, 31},
      ISSN = {2522-0160},
   MRCLASS = {11N45 (11G05 11Y40)},
  MRNUMBER = {4661854},
MRREVIEWER = {Harald A. Helfgott},
       DOI = {10.1007/s40993-023-00482-6},
}

@inproceedings {Mazurratpoints,
    AUTHOR = {Mazur, B.},
     TITLE = {Rational points on modular curves},
 BOOKTITLE = {Modular functions of one variable, {V} ({P}roc. {S}econd
              {I}nternat. {C}onf., {U}niv. {B}onn, {B}onn, 1976)},
    SERIES = {Lecture Notes in Math., Vol. 601},
     PAGES = {107--148},
 PUBLISHER = {Springer, Berlin-New York},
      YEAR = {1977},
}

@article {PPV,
    AUTHOR = {Pizzo, Maggie and Pomerance, Carl and Voight, John},
     TITLE = {Counting elliptic curves with an isogeny of degree three},
   JOURNAL = {Proc. Amer. Math. Soc. Ser. B},
  FJOURNAL = {Proceedings of the American Mathematical Society. Series B},
    VOLUME = {7},
      YEAR = {2020},
     PAGES = {28--42},
}

@article {PS,
    AUTHOR = {Pomerance, Carl and Schaefer, Edward F.},
     TITLE = {Elliptic curves with {G}alois-stable cyclic subgroups of order
              4},
   JOURNAL = {Res. Number Theory},
  FJOURNAL = {Research in Number Theory},
    VOLUME = {7},
      YEAR = {2021},
    NUMBER = {2},
     PAGES = {Paper No. 35, 19},
}

@article {Sankar,
    AUTHOR = {Boggess, Brandon and Sankar, Soumya},
     TITLE = {Counting elliptic curves with a rational {$N$}-isogeny for
              small {$N$}},
   JOURNAL = {J. Number Theory},
  FJOURNAL = {Journal of Number Theory},
    VOLUME = {262},
      YEAR = {2024},
     PAGES = {471--505},
}
\end{document}